\documentclass{amsart} 
\usepackage{amscd}
\usepackage[centertags]{amsmath}
\usepackage{newlfont}
\usepackage{graphicx}
\usepackage[usenames]{color}
\usepackage{amsfonts, amssymb}
\usepackage{mathrsfs}
\usepackage{latexsym}
\usepackage{hyperref}
\hypersetup{
colorlinks=true,
linkcolor=blue,
filecolor=blue,      
urlcolor=blue,
citecolor=blue,
}
\usepackage{verbatim}

\newtheorem{theorem}{Theorem}[section]

\newtheorem{proposition}[theorem]{Proposition}
\newtheorem{corollary}[theorem]{Corollary}
\newtheorem{lemma}[theorem]{Lemma}

\newtheorem{example}[theorem]{Example}

\theoremstyle{remark}
\newtheorem{remark}[theorem]{Remark}

\newcommand{\HH}{\mathcal{H}}

\newcommand{\M}{\mathcal{M}}

\newcommand{\PI}[2]{\left\langle \,#1 , #2\, \right\rangle}

\DeclareMathOperator{\Span}{span}

\title{Polyak's theorem on Hilbert spaces}

\author{Maximiliano Contino}
\address{Maximiliano Contino,  Instituto Argentino de Matem\'atica ``Alberto P. Calder\'on'' CONICET \\
Saavedra 15, Piso 3, (1083) Buenos Aires, Argentina  and 
Facultad de Ingenier\'{\i}a, Universidad de Buenos Aires \\
Paseo Col\'on 850 (1063), Buenos Aires, Argentina.}
\email{mcontino@fi.uba.ar}

\author{Guillermina Fongi}
\address{ Guillermina Fongi,  Centro Franco Argentino de Ciencias de la Informaci\'on y de Sistemas  CONICET \\ Ocampo y Esmeralda (2000),  Rosario, Argentina.}
\email{gfongi@conicet.gov.ar}

\author{Santiago Muro}
\address{ Santiago Muro,  Centro Franco Argentino de Ciencias de la Informaci\'on y de Sistemas CONICET \\ Ocampo y Esmeralda  (2000),  Rosario, Argentina. }
\email{muro@cifasis-conicet.gov.ar}

\begin{document}

\maketitle

\begin{abstract}
We extend to infinite dimensional Hilbert spaces
a celebrated  result, due to B. Polyak, about the convexity of the joint image of quadratic functions. We show sufficient conditions which assure that the joint image  is also closed. 
However, we prove that the closedness part of Polyak's theorem does not hold in general in the infinite dimensional setting.
Finally, we give some applications to S-lemma type results.
\end{abstract}

\section{Introduction}
In \cite{polyak1998convexity}, Polyak extended a well-known theorem of Dines \cite{dines1941mapping}, by providing a convexity property related to non-homogeneous quadratic functions. Consider the functions 
 $$
 \phi_i(x)=\langle A_ix,x\rangle +\langle x,a_i\rangle + b_i,
 $$
where $A_i$ is a $n \times n$ symmetric  matrix, $a_i\in \mathbb R^n$, $b_i\in \mathbb R$ for $i=1,2$. 
Polyak's result \cite{polyak1998convexity} states that if $n\geq 2$ and there exists $(\mu_1,\mu_2)\in \mathbb R^2$ such $\mu_1A_1 +\mu_2A_2> 0$ then the set
$$
\{(\phi_1(x), \phi_2(x)): x\in \mathbb{R}^n\}
$$ is closed and convex. 
Here, the notation $\mu_1A_1 +\mu_2A_2> 0$ means that the matrix $\mu_1A_1 +\mu_2A_2$ is positive definite.
Polyak also proved that the joint image of three homogeneous quadratic forms in $\mathbb R^n$ is a closed and convex cone of $\mathbb R^3$ if and only if there is a positive definite linear combination of the operators determining the three quadratic forms.

 In \cite{baccari2009extension} an extension of Polyak's theorems to quadratic forms defined by compact operators on infinite dimensional separable Hilbert spaces was investigated. However, in \cite[Theorems 2.1 and 2.3]{baccari2009extension}, some compact operators are assumed to be bounded below,
so unfortunately,their main results are only applicable to finite dimensional spaces (see the comments after Corollary \ref{Brickman1bis}).
 Moreover, Example \ref{ejemplo no cerrado} shows that the joint image can be non-closed, even for quadratic functions determined by compact  positive definite operators. This shows  that additional hypothesis must be considered in order to prove the closedness part of Polyak's theorem.

In this work we extend Polyak's convexity result to an arbitrary infinite dimensional Hilbert space $\HH$. Moreover, we show that if $A_1$ is a compact operator on $\HH$ with 0 in its numerical range and $A_2$ is a positive invertible definite operator, then the joint image of two non necessarily homogeneous quadratic forms determined by $A_1$ and $A_2$, is also closed. We finish this work with some applications to S-lemma type results.


\section{Extension of Polyak's results: the homogeneous case}
In this section we prove the convexity of the joint image of three homogeneous quadratic forms on a Hilbert space. Let us first introduce some notations.

Throughout $\HH$ and $\mathcal K$ denote real inner product spaces.  The  range and nullspace of any given mapping $A$ are denoted by $R(A)$ and $N(A),$ respectively. 
Also, $L(\HH, \mathcal K)$ stands for the space of the bounded linear operators defined on $\HH$ to $\mathcal K.$ When $\HH = \mathcal K$ we write, for short, $L(\HH).$ 
Given a linear operator $T$ on $\HH$ (possibly densely defined) we say that $T$ is \emph{positive definite} or $T >0$  if $T$ is \emph{symmetric} (i.e., $\PI{Tx}{y}=\PI{x}{Ty} \mbox{ for every } x, y$ in the domain of $T$) and  
$\PI{Tx}{x} >0 \mbox{ for every } x\ne0$ in the domain of $T$. 
The group of invertible operators in $L(\HH)$ is denoted by $GL(\HH)$ and  $GL(\HH)^+$ denotes the set of positive definite and invertible operators in $L(\HH)$.
For a closed subspace $\M,$ $P_{\M}$ denotes the orthogonal projection onto $\M.$ Finally, $S_\HH$ and $B_\HH$ denote the unit sphere and the open unit ball of $\HH$, respectively.

\medskip

A key tool used in the proof of Polyak's theorems is a result on the joint real numerical range of real symmetric matrices due to Brickman \cite{brickman1961field}. This result can be seen as the real analogue of the classical Toeplitz-Hausdorff Theorem (and implies it, see e.g. \cite{Mcintosh}). Brickman's result was extended to infinite dimensional inner product spaces,
\cite{micchelli1993optimal,martinez2005brickman} (see also \cite[Theorem 2]{hestenes1968pairs}):
\begin{theorem}[Brickman's convexity]\label{Brickman1}
	Let $(\HH, \PI{\cdot}{\cdot})$ be a real inner product space, $3\le dim(\HH)\le \infty$. Let $A_1,A_2$ be (not necessarily bounded) linear endomorphisms on $\HH$. Then the set
	 \begin{align*}
	W_{\mathbb R}(A_1,A_2):=\{(\langle A_1x,x\rangle,\langle A_2x,x\rangle)\in \mathbb R^2\,:\, \|x\|=1\}
	\end{align*}
	is a convex subset of $\mathbb R^2$.
\end{theorem}
As a consequence of Brickman's convexity theorem, it is easy to show that a similar result holds considering two different  inner products in $\HH.$
 The following corollary will be useful to prove our convexity result (see Theorem  \ref{convexidad Polyak1}). In order to include examples of densely defined unbounded operators  (e.g. the differentiation operator on $L^2(\mathbb R)$) we state the next corollary for linear mappings from an inner product space to its completion.  

\begin{corollary}\label{Brickman1bis}
	Let $\HH$ be a real vector space and  let $\langle \cdot , \cdot \rangle, \langle \cdot , \cdot \rangle_*$  be two inner products on $\HH$ and  $3\le dim(\HH)\le \infty$. Consider $A_1,A_2$  (not necessarily bounded) linear transformations from $\HH$ to $\tilde\HH,$ where $\tilde \HH$ denote the completion of $\HH$ with respect to the inner product $\PI{\cdot}{\cdot}$. Then the set
	 \begin{align*}
	\{(\langle A_1x,x\rangle,\langle A_2x,x\rangle)\in \mathbb R^2\,:\, \|x\|_*=1\}
	\end{align*}
	is a convex subset of $\mathbb R^2$, where $\|\cdot\|_*$ is the norm associated to the inner product $\langle \cdot , \cdot \rangle_*$.
\end{corollary}
{\it Proof.}  As in the proof of \cite[Theorem 2.2]{martinez2005brickman}, we first consider   $\HH=\mathbb{R}^{3}.$ In this case $\HH=\tilde\HH=\mathbb{R}^{3}$ and, since $\langle \cdot , \cdot \rangle$ is continuous on $(\mathbb R^3,\langle \cdot , \cdot \rangle_*)$, 
there exists $B \in L(\mathbb{R}^{3})$ such that  $\langle x,y\rangle=\langle Bx,y\rangle_*$ for every $x,y \in \mathbb{R}^3.$
Thus, by Theorem \ref{Brickman1} for $\HH=\mathbb{R}^{3},$  the set
$$\{(\langle A_1x,x\rangle,\langle A_2x,x\rangle)\in \mathbb R^2\,:\, \|x\|_*=1\}=\{(\langle BA_1x,x\rangle_*,\langle BA_2x,x\rangle_*)\in \mathbb R^2\,:\, \|x\|_*=1\}$$
is convex.

Now suppose that $3 \leq \dim(\HH) \leq \infty.$ 

Let $y_1:=(\PI{A_1x_1}{x_1},\PI{A_2x_1}{x_1})$ and $y_2:=(\PI{A_1x_2}{x_2},\PI{A_2x_2}{x_2})$, with $\Vert x_1 \Vert_* = \Vert x_2 \Vert_* = 1,$ be any two different points in $\{(\langle A_1x,x\rangle,\langle A_2x,x\rangle)\in \mathbb R^2\,:\, \|x\|_*=1\}.$ Take any orthonormal basis $\{w_1, w_2\}$ of the space $(\Span\{x_1, x_2\}, \PI{\cdot}{\cdot})$ and take another vector $w_3$ such that $\PI{w_3}{w_1}=\PI{w_3}{w_2}=0$ and $\Vert w_3 \Vert_*= 1.$ Set $W:=\Span\{w_1, w_2, w_3\}$ and consider  the operators $\tilde A_l:=P_W(A_l){|_W}:W\to W$, for $l=1,2$. Then by the first part of the proof, $\{(\langle \tilde{A}_1x,x\rangle,\langle \tilde{A_2}x,x\rangle) \, :\, x \in W, \|x\|_*=1\}$
is convex.

Moreover, since  $\langle \tilde{A}_l x,x\rangle=\langle {A}_l x,x\rangle$ for any $x\in W,$ we have that
$$
\{(\langle \tilde{A}_1x,x\rangle,\langle \tilde{A_2}x,x\rangle) \, :\, x \in W, \|x\|_*=1\}\subset \{(\langle A_1x,x\rangle,\langle A_2x,x\rangle):\,x\in\HH,\, \|x\|_*=1\}.
$$
Finally, since  $x_1,x_2\in W$, we conclude that 
 for every $\lambda \in [0,1],$ $$(1-\lambda)y_1+\lambda y_2 \in 	\{(\langle A_1x,x\rangle,\langle A_2x,x\rangle)\in \mathbb R^2\,:\, \|x\|_*=1\}.$$\qed

\medskip
In \cite{baccari2009extension} the authors tried to extend Polyak's theorems to quadratic forms defined by compact operators on infinite dimensional separable real Hilbert spaces. 
For example, Theorem 2.1 in \cite{baccari2009extension} was intended to show the closedness part of Polyak's theorem. There it is assumed  that   $A_1,A_2 \in L(\HH)$ are compact operators and that there exist scalars $\mu_1,\mu_2 \in \mathbb{R}$ such that $C:=\mu_1A_1+\mu_2 A_2$ satisfies that for some $\alpha >0$,
\begin{equation}\label{bounded below}
 \PI{Cx}{x} \geq \alpha \Vert x \Vert^2 \mbox{ for every } x \in \HH.
\end{equation} 
The reader should be aware that in \cite{baccari2009extension}  an operator $C$ satisfying \eqref{bounded below} is denoted by $C>0$ . It is well known that there are no compact operators on infinite dimensional Hilbert spaces that satisfy \eqref{bounded below}. Indeed, consider $(x_n)_{n \geq 1} \subseteq \overline{B_{\HH}}$ (the closed unit ball). Since $(x_n)_{n \geq 1}$ is bounded and $\overline{B_{\HH}}$ is a closed subset of $\HH,$ then there exists a subsequence $(x_{n_k})_{k \geq 1} \subseteq \overline{B_{\HH}}$  and $x_0 \in\overline{B_{\HH}}$ such that 
$(x_{n_k})_{k \geq 1}$ converges weakly to $x_0.$  Since $C$ is compact, it follows that  $\underset{k \rightarrow \infty}{\lim} \Vert Cx_{n_k} - Cx_0\Vert=0.$ Therefore
\begin{align*}
\Vert  x_{n_k} - x_0 \Vert^2 &\leq \frac{1}{\alpha} \PI{C(x_{n_k}-x_0)}{x_{n_k}-x_0} \leq \frac{1}{\alpha} \Vert Cx_{n_k} - Cx_0\Vert \Vert x_{n_k} - x_0\Vert\underset{k \rightarrow \infty}{\longrightarrow}0.
\end{align*}
Then $\overline{B_{\HH}}$ is norm compact. Therefore $\HH$ is finite dimensional. 

\medskip

The following examples show that the closedness part of Polyak's theorem does not hold neither for pairs of compact positive definite operators (Example \ref{ejemplo no cerrado}) nor for pairs of bounded below operators (Example \ref{ejemplo no cerrado2}) on infinite dimensional spaces. Also, it is not  difficult to extend  both examples  to $k$-tuples of operators.



\begin{example}\label{ejemplo no cerrado}
Take any sequence $(\alpha_n)_n$ of positive real numbers converging to $0$. Consider on $\ell_2$  (the usual Hilbert space of square summable sequences with orthonormal basis $(e_n)_{n\in\mathbb N}$) a pair of diagonal operators defined by
$$
A_0(x)=(\alpha_nx_n)_n,\quad A_1(x)=\big(\alpha_n(1+\frac{1}{n})x_n\big)_n.
$$
Note that $A_0,A_1$ are both positive definite. Since $\alpha_n\to 0$, then $A_0,A_1$ can be uniformly approximated  by finite range operators, so they are both compact operators. Moreover, note that 
for $j=0,1$,
$$ 
\langle A_j\alpha_n^{-1/2}e_n,\alpha_n^{-1/2}e_n\rangle =\alpha_n^{-1}\langle A_je_n,e_n\rangle = 1+\frac{j}{n} \to 1, \quad \textrm{as }n\to\infty.
$$
This means that $(1,1)$ is in the closure of 
$\{(\langle A_0x,x\rangle,\langle A_1x,x\rangle):\,x\in\ell_2 \}$ in $\mathbb R^2$.

But on the other hand, $(1,1)\notin\{(\langle A_0x,x\rangle,\langle A_1x,x\rangle):\,x\in\ell_2 \}$ because
for any $x\ne 0,$
\begin{align*}
\langle A_0x,x\rangle =\sum_n\alpha_nx_n^2<\sum_n\alpha_n(1+\frac1{n})x_n^2=\langle A_1x,x\rangle.
\end{align*}
Therefore the image of the quadratic form determined by $A_0,A_1$ is not closed.
\end{example}

\begin{example}\label{ejemplo no cerrado2}
Take any sequence $(\alpha_n)_n$ of positive real numbers converging to $\alpha>0$. As in Example \ref{ejemplo no cerrado}, let $A_0,A_1$ be operators on $\HH=\ell_2$ defined by,
$$
A_0(x)=(\alpha_nx_n)_n,\quad A_1(x)=\big(\alpha_n(1+\frac{1}{n})x_n\big)_n.
$$
Then $A_0, A_1 \in GL(\HH)^+$  because $\alpha>0$ (in particular, both operators satisfy \eqref{bounded below}). Moreover, proceeding as in the previous example, it follows that $(1,1)$ is in the closure of the image
$\{(\langle A_0x,x\rangle,\langle A_1x,x\rangle):\,x\in\ell_2 \}$ in $\mathbb R^2$, but not in the image of the quadratic form determined by $A_0,A_1$. 
\end{example}

\begin{remark}
It is known that the numerical range of a compact operator is not necessarily closed on infinite dimensional Hilbert spaces: take for example on $\ell_2$ the operator $(x_n)_n\mapsto (\frac{x_n}{n})_n$, then the numerical range is $(0,1]$ (see \cite[Problem 212]{halmos}).
Thus, the image of the unit sphere by pairs of quadratic forms (i.e. the joint numerical range) is not closed in general for infinite dimensional spaces, even for compact operators. On the other hand, since a quadratic form  determined by a compact operator is weakly continuous on bounded sets, and the closed unit ball is weakly compact, we immediately conclude the following: given $\{A_1,\cdots,A_n\},$ any collection of compact operators, the set
$$
\{(\langle A_1x,x\rangle,\cdots,\langle A_nx,x\rangle)\in \mathbb R^n\,:\, \Vert x \Vert \leq 1 \}\quad 
$$ 
\textrm{ is closed}.

\end{remark}

\bigskip
Next, we give some conditions under which the joint image of three quadratic forms is closed and convex. First we need the following lemma, which is an extension of a result in \cite{DeBarra} and shows, using the same ideas, that under certain conditions the joint numerical range of compact operators on real Hilbert spaces is closed.
\begin{lemma}\label{rango numerico compacto}
Consider $A_1,A_2$  compact selfadjoint operators on a real Hilbert space $\HH$. Suppose that $(0,0)\in W_{\mathbb{R}}(A_1,A_2)$ then $W_{\mathbb R}(A_1,A_2)$ is closed. 
\end{lemma}
\begin{proof}
    Let $\lambda\in \overline{W_{\mathbb R}(A_1,A_2)}$. Since the closed ball is weakly compact then $$\lambda=\lim_\alpha (\langle A_1x_\alpha,x_\alpha\rangle,\langle A_2x_\alpha,x_\alpha\rangle)$$ for some net $(x_\alpha)_\alpha$ with $\|x_\alpha\|=1$ weakly convergent to some $x$ with $\|x\|\le1$. Moreover, since the operators $A_1,A_2$ are compact, it is easy to see that $\lambda=(\langle A_1x,x\rangle,\langle A_2x,x\rangle)$. If $\lambda=(0,0)$ there is nothing to prove. Otherwise, $x\ne0,$ and thus $\frac{\lambda}{\|x\|^2}$ belongs to $W_{\mathbb R}(A_1,A_2)$. Finally, since $\|x\|\le1$ and $(0,0)\in W_{\mathbb R}(A_1,A_2),$ we conclude that $\lambda\in W_{\mathbb R}(A_1,A_2)$ by Theorem \ref{Brickman1}.  \qed
    \end{proof}

\begin{theorem} \label{propclosed}
Let $F(x)=(\langle A_1x,x\rangle,\langle A_2x,x\rangle,\langle A_3x,x\rangle)$ be a quadratic mapping determined by bounded operators $A_1,A_2,A_3$ on a real Hilbert space $\HH$. Suppose that there exist $\mu_1, \mu_2, \mu_3 \in \mathbb{R}$ such that $\mu_1A_1+\mu_2A_2+\mu_3A_3 \in GL(\HH)^+,$  $A_1,A_2$ are compact and  $(0,0)\in W_{\mathbb{R}}(A_1,A_2)$. Then $F(\HH)$ is closed.
\end{theorem}
\begin{proof}
We may assume that $A_1,A_2,A_3$ are selfadjoint and we assume that $\HH$ is infinite dimensional because the finite dimensional case was proved by Polyak  \cite[Theorem 2.1]{polyak1998convexity}.

We first assert that it is sufficient to prove the case when $A_1,A_2$ are compact and $A_3=I$. In fact, consider the linear transformation  $T: \mathbb R^3 \rightarrow \mathbb R^3$ defined by $T(r,s,t)=(r,s,\mu_1 r+\mu_2 s+ \mu_3 t).$ Since $\mu_3 \not =0,$ $T$ is invertible and preserve closedness. Then it suffices to prove that $$T(F(\mathcal H))=\{(\langle  A_1x,x\rangle,\langle  A_2x,x\rangle,\langle \tilde A_3x,x\rangle)\in \mathbb R^3\,:\, x\in\HH\}$$ is closed, where $\tilde A_3:=\mu_1A_1+\mu_2 A_2+\mu_3 A_3$.

Since $\tilde A_3 \in GL(\HH)^+,$ the inner product $\langle x,y\rangle_*:=\langle\tilde A_3x,y\rangle$ makes $(\HH,\langle \cdot, \cdot\rangle_*)$ a Hilbert  space. Denote by $\|\cdot\|_*$ the induced norm, which is equivalent to $\|\cdot\|$.

Then $\PI{x}{y}=\PI{\tilde A_3 \tilde A_3^{-1}x}{y}=\langle \tilde A_3^{-1}x, y \rangle_*,$ for every $x,y \in \HH,$ and 
$$T(F(\HH))=\{(\langle \tilde A_3^{-1}A_1x,x\rangle_*,\langle \tilde A_3^{-1}A_2x,x\rangle_*,\|x\|_*^2)\in \mathbb R^3\,:\, x\in\HH\}.$$ 

Finally note that, in $(\HH,\langle\cdot,\cdot\rangle_*),$ we have that $\tilde A_3^{-1}A_1,\tilde A_3^{-1}A_2$ are compact operators and $(0,0)\in W_{\mathbb{R}}(\tilde A_3^{-1}A_1,\tilde A_3^{-1}A_2)$.

Suppose then that $A_3=I$ and take $\lambda=(\lambda_1,\lambda_2,\lambda_3)\in \overline{F(\HH)}$. 
Then $$\lambda=\lim_n F(x_n)=\lim_n (\langle A_1x_n,x_n\rangle,\langle A_2x_n,x_n\rangle,\Vert x_n \Vert^2)$$ for some sequence $(x_n)_n \subseteq \HH.$ If $\lambda_3=0$ then $0=\lambda_3=\lim_n \Vert x_n \Vert^2.$ So that  $\lambda=0\in F(\HH).$ 

If $\lambda_3\ne0$, then $\lambda_3=\lim_n \Vert x_n \Vert^2.$ Therefore,
$$\lim_n \langle A_j\frac{x_n}{\Vert x_n \Vert},\frac{x_n}{\Vert x_n \Vert}\rangle =\frac{\lambda_j}{\lambda_3} \mbox{ for } j=1, 2.$$
Then $(\frac{\lambda_1}{\lambda_3},\frac{\lambda_2}{\lambda_3},1)\in \overline{F(S_\HH)}$ and $(\frac{\lambda_1}{\lambda_3},\frac{\lambda_2}{\lambda_3})\in \overline{W_{\mathbb R}(A_1,A_2)}=W_{\mathbb R}(A_1,A_2),$ where we used Lemma \ref{rango numerico compacto}. Hence, there exists $z \in S_\HH$ such that $$(\frac{\lambda_1}{\lambda_3},\frac{\lambda_2}{\lambda_3})=(\langle A_1z,z\rangle,\langle A_2z,z\rangle).$$ Then
$F(\lambda_3^{1/2}z)=(\lambda_3 \langle A_1z,z\rangle,\lambda_3 \langle A_2z,z\rangle, \lambda_3)=\lambda,$
so that $\lambda\in F(\HH)$.\qed
\end{proof}
\begin{remark}
    Modifying Example \ref{ejemplo no cerrado}, it  can be  seen that the assumption $(0,0)\in W_{\mathbb{R}}(A_1,A_2)$ cannot be dropped in the above theorem. Indeed, take $A_j(x)=(\frac{j}{n}x_n)_n,$ for $j=1,2,$ $A_3=I$. Then $(0,0,1)=\lim_n F(e_n)$ is in $\overline{F(\HH)}$ but not in $F(\HH)$.

\end{remark}

With a similar proof we may show the following more general result.
\begin{corollary} \label{prop homog closed gral}
Let $F(x)=(\langle A_1x,x\rangle,\langle A_2x,x\rangle,\langle A_3x,x\rangle)$ be a quadratic mapping determined by operators $A_1,A_2,A_3$ on a real Hilbert space $\HH$. 
Suppose that there are linear combinations
$$\tilde A_i:=\mu_{i1}A_1+\mu_{i2}A_2+\mu_{i3}A_3 \mbox{ for } i=1,2,3$$
such that the $3\times 3$ matrix of real numbers  $\displaystyle \mu=(\mu_{ij})_{i,j=1}^3$ is not singular, $\tilde A_3 \in GL(\HH)^+,$ $\tilde A_1,\tilde A_2$ are compact and  $(0,0)\in W_{\mathbb{R}}(\tilde A_1,\tilde A_2)$. Then $F(\HH)$ is closed.
\end{corollary}


We prove now the extension of Polyak convexity theorem \cite[Theorem 2.1]{polyak1998convexity} to not necessarily bounded linear operators on inner product spaces.

\begin{theorem}\label{convexidad Polyak1}
	Let $\HH$ be a real inner product space, $3\le dim(\HH)\le \infty$. Let $A_1,A_2,A_3$ be linear transformations from $\HH$ to its completion $\tilde \HH$  such that there exist $\mu_1, \mu_2, \mu_3 \in \mathbb{R}$  with $\mu_1A_1+\mu_2 A_2+\mu_3 A_3>0$. Then the set
	\begin{align*}
	F(\HH)=\{(\langle A_1x,x\rangle,\langle A_2x,x\rangle,\langle A_3x,x\rangle)\in \mathbb R^3\,:\, x\in\HH\}
	\end{align*}
	is a convex cone in $\mathbb R^3$. 
\end{theorem}

\begin{proof}
We may suppose that $\mu_3\ne 0$ (otherwise we interchange the order of the operators).

As in the proof of Theorem \ref{propclosed}, consider the linear transformation  $T: \mathbb R^3 \rightarrow \mathbb R^3$ defined by $T(r,s,t)=(r,s,\mu_1 r+\mu_2 s+ \mu_3 t).$ Then $T$ is invertible and preserves convexity. Therefore it suffices to prove that 
$$T(F(\mathcal H))=\{(\langle  A_1x,x\rangle,\langle  A_2x,x\rangle,\langle \tilde A_3x,x\rangle)\in \mathbb R^3\,:\, x\in\HH\}$$ is convex, where $\tilde A_3:=\mu_1A_1+\mu_2 A_2+\mu_3 A_3$. 
Since $\tilde A_3>0$, the bilinear form $\PI{\cdot}{\cdot}_{*}:=\langle\tilde A_3\cdot,\cdot\rangle$ makes $\HH_*:=(\HH,\PI{\cdot}{\cdot}_{*})$ an inner product space with norm denoted by $\|\cdot\|_*$.
Then $$T(F(\HH))=\{(\langle A_1x,x\rangle,\langle A_2x,x\rangle,\|x\|_*^2)\in \mathbb R^3\,:\, x\in\HH\}.$$

By Corollary \ref{Brickman1bis}, the set $\{(\langle A_1x,x\rangle,\langle A_2x,x\rangle,\|x\|_*^2)\in \mathbb R^3\,:\, \|x\|_*=1\}=T(F(S_{\HH_*}))$ is convex. Hence by homogeneity, $T(F(\mathcal H))$ is a convex cone because $$T(F(\HH))=\bigcup_{t\ge0}t\cdot \{(\langle A_1x,x\rangle,\langle A_2x,x\rangle,\|x\|_*^2)\in \mathbb R^3\,:\, \|x\|_*=1\}.$$\qed
\end{proof}

\section{The non-homogeneous case}

Using the closedness of the joint image of a pair of non necessarily homogeneous quadratic forms, it was proved in \cite[Theorem 2.2]{baccari2009extension} that this image is also convex. 
We will now prove Polyak's theorem for non-homogeneous quadratic forms without assuming that it is closed.

\begin{proposition}\label{Polyakdiminfinita}
Let $\HH$ be a real inner product space, $3\le dim(\HH)\le \infty$. 
Let $A_1,A_2\in L(\HH)$ be such that $\mu_1 A_1+\mu_2 A_2>0$ for some $\mu_1,\mu_2\in \mathbb R,$  $a_1,a_2\in\HH$ and $b_1,b_2\in\mathbb R$. 
Let $\Phi=(\phi_1,\phi_2)$ be the non-homogeneous quadratic form defined by 
$\phi_j(x)=\langle A_jx,x\rangle +\langle x,a_j\rangle + b_j$, $j=1,2.$  Then 
\begin{align*}
\Phi(\HH)=\{(\phi_1(x),\phi_2(x))\in\mathbb R^2\,:\, x\in\HH\}
\end{align*}
is convex.
\end{proposition}
\begin{proof} 
Let $t,s\in \Phi(\HH),$ with $t\ne s$, then there exist $x, y \in \HH$ such that 
$$t=\Phi(x) \mbox{ and } s=\Phi(y).$$
Consider $\tilde\HH:=\Span \{w,x, y\}$, where $w\in\HH$ is linearly independent to $x$ and $y$. Note that $2\le dim(\tilde\HH)\le 3$. 
Let $\PI{\cdot}{\cdot}_{\tilde \HH}$ be the restriction of $\PI{\cdot}{\cdot}$ to $\tilde \HH.$ Let $P_{\tilde \HH}$ denote the orthogonal projection onto the  finite dimensional Hilbert space $\tilde\HH.$
Set  $\tilde{\Phi}:=\Phi|_{\tilde\HH}=(\tilde{\phi_1},\tilde{\phi_2})$ where $\tilde{\phi_j}:=\phi_j|_{\tilde\HH}$ for $j=1,2.$ 
Then, $\tilde{\phi_j}:\tilde{\HH} \rightarrow \mathbb R,$
$t=\tilde{\Phi}(x),$ $s=\tilde{\Phi}(y)$ and, for $z \in \tilde \HH,$
\begin{align*}
\tilde{\phi_j}(z)&=\PI{A_jz}{z}+\PI{a_j}{z}+b_j\\
&=\PI{A_j|_{\tilde\HH}z}{P_{\tilde \HH}z}+\PI{a_j}{P_{\tilde \HH}z}+b_j\\
&=\PI{P_{\tilde \HH}A_j|_{\tilde\HH}z}{z}_{\tilde \HH}+\PI{P_{\tilde \HH}a_j}{z}_{\tilde \HH}+b_j.
\end{align*}
Let $\tilde{A_j}:=P_{\tilde \HH}A_j|_{\tilde\HH}$ for $j=1,2.$ Then  $\mu_1 \tilde{A_1}+\mu_2 \tilde{A_2}>0.$ In fact, for $z \in \tilde \HH$ we have
\begin{align*}
\PI{(\mu_1 \tilde{A_1}+\mu_2 \tilde{A_2})z}{z}_{\tilde \HH}&=\PI{(\mu_1P_{\tilde \HH}A_1|_{\tilde\HH}+\mu_2P_{\tilde \HH}A_2|_{\tilde\HH})z}{z}=\PI{(\mu_1A_1+\mu_2A_2)z}{z} > 0.
\end{align*}
Then, by Polyak's Theorem, $\tilde \Phi(\tilde\HH)$ is a convex set. Therefore, for every $\alpha \in [0,1],$ 
$$\alpha t+ (1-\alpha) s \in \tilde{\Phi}(\tilde \HH) \subseteq \Phi(\HH).$$  
Hence $\Phi(\HH)$ is a convex set.\qed
\end{proof}

\begin{remark}
We may actually prove the convexity of the image of $\Phi$ under the hypothesis of $A_1,A_2$ being non-degenerate (that is, if $\PI{A_1u}{u}=0=\PI{A_2u}{u}$ then $u=0$).
For infinite dimensional Hilbert spaces, this is a strictly weaker assumption, see e.g. \cite{calabi1964linear}.
\end{remark}
\begin{proof}
Using the notation as in the proof of Proposition \ref{Polyakdiminfinita}, it is clear that $\tilde A_1,\tilde A_2$ is a non-degenerate pair.
If the 2-homogeneous part of $\tilde\Phi$ is not surjective, then
by \cite[Corollary 1]{dines1941mapping}, there are $\mu_1, \mu_2 \in \mathbb{R}$ such that $\mu_1 \tilde{A_1}+\mu_2 \tilde{A_2}>0.$ Then, by Polyak's Theorem, $\tilde \Phi(\tilde\HH)$ is a convex set.

On the contrary, if the 2-homogeneous part of $\tilde\Phi$ is surjective, then by \cite[Lemma 4.10]{flores2016characterizing}, $\tilde \Phi(\tilde \HH)=\mathbb R^2.$
Therefore,  $\tilde \Phi(\tilde\HH)$ is a convex set. Then, for every $\alpha \in [0,1],$ 
Therefore $\hat{A}_1$ and $\hat{A}_3$ are compact operators and $(0,0)\in W_{\mathbb{R}}(\hat{A}_1, \hat{A}_3).$
Hence $\Phi(\HH)$ is a convex set.
\qed
\end{proof}

\begin{proposition} \label{propclosednohomogeneo}
Let $\HH$ be a real Hilbert space, $3\le dim(\HH)\le \infty$. Let $A_1,A_2\in L(\HH)$ be selfadjoint operators, $a_1,a_2\in\HH$ and $b_1,b_2\in\mathbb R$. 
Let $\Phi=(\phi_1,\phi_2)$ be the non-homogeneous quadratic form defined by 
$\phi_j(x)=\langle A_jx,x\rangle +\langle x,a_j\rangle + b_j$, $j=1,2.$ 
Suppose that there are linear combinations 
$$
\tilde A_1:=\alpha_1A_1+\alpha_2A_2, \quad \tilde A_2 :=\beta_1A_1+\beta_2A_2
$$ 
such that $\alpha_1\beta_2-\alpha_2\beta_1\ne 0$,
$\tilde A_1$ is compact, $\langle \tilde A_1x,x\rangle=0$ for some $x \ne 0$ and $\tilde A_2 \in GL(\HH)^+.$
Then $\Phi(\HH)$ is convex and closed.
\end{proposition}

In particular if $0$ is in the numerical range of $A_1$, $A_1$ is compact and $A_2 \in GL(\HH)^+.$
Then $\Phi(\HH)$ is convex and closed.

\begin{proof} Since $\tilde A_2 \in GL(\HH)^+,$ by Proposition \ref{Polyakdiminfinita}, $\Phi(\HH)$ is convex.

Now we are going to show that $\Phi(\HH)$ is closed. As in the proof of Theorem \ref{propclosed}, consider the linear transformation  $T: \mathbb R^2 \rightarrow \mathbb R^2$ defined by $T(r,t)=(\alpha_1 r+\alpha_2 t,\beta_1 r+\beta_2 t).$ Since $\alpha_1\beta_2-\alpha_2\beta_1\ne 0,$
$T$ is invertible and preserves closedness. 
Therefore it suffices to prove that 
$$T(\Phi(\mathcal H))=\{(\langle  \tilde A_1x,x\rangle+\langle x,\tilde a_1\rangle + \tilde b_1,\langle  \tilde A_2x,x\rangle+\langle x,\tilde a_2\rangle + \tilde b_2)\in \mathbb R^2\,:\, x\in\HH\}$$ is closed, where $\tilde A_1:=\alpha_1A_1+\alpha_2 A_2$ is compact, $0=\langle \tilde A_1x, x\rangle$ for some $x \in S_\HH,$ $\tilde A_2:=\beta_1A_1+\beta_2 A_2  \in GL(\HH)^+$, $\tilde a_1:=\alpha_1a_1+\alpha_2a_2,$ $\tilde a_2 :=\beta_1a_1+\beta_2a_2,$ $\tilde b_1:=\alpha_1b_1+\alpha_2b_2$ and $\tilde b_2 :=\beta_1b_1+\beta_2b_2$.

Let $\tilde\HH:=\HH\times\mathbb R$ and define the following 2-homogeneous forms on $\tilde\HH$:
\begin{align*}
f_j(x,t)= & \langle \tilde A_jx,x\rangle +t\langle x,\tilde a_j\rangle + t^2 \tilde b_j,\qquad j=1,2\\ 
f_3(x,t)= & t^2.
\end{align*} Then, the homogeneous quadratic form $f_j$ is determined by the selfadjoint operators 
$$\hat{A}_j:=\left(\begin{array}{cc}
\tilde A_j & \frac{\tilde a_j}{2}\\
\langle\cdot,\frac{\tilde a_j}{2}\rangle & \tilde b_j
\end{array}\right) \mbox{ for } j=1,2 \mbox{ and } \hat{A}_3=\left(\begin{array}{cc}
0 & 0\\
0 & 1
\end{array}\right).$$ 

In fact, we have $\displaystyle \PI{\hat{A}_3 \left(\begin{array}{cc}
x \\
t
\end{array}\right)}{\left(\begin{array}{cc}
x \\
t
\end{array}\right)}=t^2=f_3(x,t), $  and for $j=1,2,$
\begin{align*}
\PI{\hat{A}_j \left(\begin{array}{cc}
x \\
t
\end{array}\right)}{\left(\begin{array}{cc}
x \\
t
\end{array}\right)} &=\PI{\left(\begin{array}{cc}
\tilde A_jx+ t\frac{\tilde a_j}{2}\\
\langle x,\frac{\tilde a_j}{2}\rangle+t \tilde b_j
\end{array}\right)}{\left(\begin{array}{cc}
x \\
t
\end{array}\right)}= \langle \tilde A_j x,x\rangle +t\langle x,\tilde a_j\rangle + t^2 \tilde b_j=f_j(x,t).
\end{align*}

Also, $\hat{A}_1$ and $\hat{A}_3$ are compact operators and $(0,0)\in W_{\mathbb{R}}(\hat{A}_1, \hat{A}_3).$

Let $\mu_3 \in \mathbb R$ be such that $$\mu_3 > \Vert {\tilde A_2}^{-1/2} \frac{\tilde a_2}{2} \Vert^2 - \tilde b_2$$ then $\hat{A}_2+\mu_3 \hat{A}_3 \in GL(\tilde \HH)^+.$ In fact, 
$$Z:=\hat{A}_2+\mu_3 \hat{A}_3=\left(\begin{array}{cc}
\tilde A_2 & \frac{\tilde a_2}{2}\\
\langle\cdot,\frac{\tilde a_2}{2}\rangle & \, \, \tilde b_2+\mu_3
\end{array}\right)=\left(\begin{array}{cc}
\tilde A_2 & d \\
d^* & \, \, \tilde b_2+\mu_3
\end{array}\right),$$
where  $d: \mathbb R \rightarrow \HH$ is the operator defined by 
$d(t):=t\frac{\tilde a_2}{2}.$ Then $d^*=\PI{\cdot}{\frac{\tilde a_2}{2}},$ 
$d={\tilde A_2}^{1/2}({\tilde A_2}^{-1/2}d)$ and $g:={\tilde A_2}^{-1/2}d$ is the (reduced) solution of the equation $d=\tilde {A_2}^{1/2}z,$ see \cite{Douglas}. Then
$$g^*g=d^*\tilde {A_2}^{-1}d=\Vert \tilde{A_2}^{-1/2} \frac{\tilde a_2}{2} \Vert^2.$$
Hence, $\tilde b_2+\mu_3=g^*g+t$ with $t:=\tilde b_2+\mu_3-\Vert {\tilde A_2}^{-1/2} \frac{\tilde a_2}{2} \Vert^2>0.$ Then, by \cite[Theorem 3]{anderson1975shorted}, $Z=\hat{A}_2+\mu_3 \hat{A}_3\geq 0.$ Also, 
$z:=\tilde b_2+\mu_3-g^*g=\tilde b_2+\mu_3-d^*{\tilde A_2}^{-1}d=t>0.$
Then $z^{-1}=(\tilde b_2+\mu_3-d^*{\tilde A_2}^{-1}d)^{-1} \in \mathbb R$ and, it can be checked that
$$Z^{-1}=\left(\begin{array}{cc}
{\tilde A_2}^{-1}+{\tilde A_2}^{-1}dz^{-1}d^*{\tilde A_2}^{-1}& -{\tilde A_2}^{-1}dz^{-1}\\
-z^{-1}d^*{\tilde A_2}^{-1}  & z^{-1}
\end{array}\right) \in L(\tilde \HH).$$
Therefore $Z=\hat{A}_2+\mu_3 \hat{A}_3 \in GL(\tilde \HH)^+.$

Set $F:=(f_1,f_2,f_3)$. Then, by Theorem \ref{propclosed}, $F(\tilde \HH)$ is closed. Then
$$
F(\tilde\HH)\cap\{(a,b,c)\in\mathbb R^3: c=1\}=F(\HH\times \{-1,1\})=F(\HH \times \{1\}),$$ where we used that $F(x,-1)=F(-x,1)$ for every $x \in \HH.$
Therefore, the set $F(\HH\times \{1\})$ is closed because the set $\{(a,b,c)\in\mathbb R^3: c=1\}$ is closed.
Finally, note that the projection of $F(\HH\times \{1\})$ to $\mathbb R^2$ is exactly $\Phi(\HH)$.
\qed
\end{proof}

\section{Applications}
Let $\HH$ be a real Hilbert space, $A \in L(\HH),$ $b\in \HH$ and $\rho>0.$ Consider the function $G: \HH \rightarrow \mathbb{R}$ given by 
$$G(x):=\frac{\Vert Ax-b \Vert^2}{1+\Vert x \Vert^2}+\rho \Vert x \Vert^2.$$
In \cite[Proposition 4.13]{TLS}, we apply the following version of an S-lemma in order to give a method for finding the infimum of $G.$ In that work, we give a characterization of such infimum and we present sufficient conditions for the existence of solution of a related total least squares problem.

\begin{lemma}\label{new S-lemma} Let $\HH$ be a real Hilbert space. Let $\phi_j(x)=\langle A_jx,x\rangle+\langle x,a_j\rangle+b_j$, with $A_j\in L(\HH)$, $a_j\in\HH$, $b_j\in\mathbb R$, $j=1,2.$
Suppose that $\mu_1 A_1+\mu_2 A_2>0$  for some $\mu_1,\mu_2\in\mathbb R$.
Let
$F:\mathbb R^2\to\mathbb R$ be defined as
$$
F(z)=\langle \Theta z,z\rangle+\langle z,v\rangle -t,
$$	
where $\Theta$ is a real symmetric nonnegative $2\times 2$ matrix, $v=(v_1,v_2)\in\mathbb R^2$ and $t\in\mathbb R$. Then the following are equivalent:
\begin{itemize}
	\item[(i)]  $F(\phi_1(x),\phi_2(x))\ge 0$ for every $x\in\HH$.
	
	\item[(ii)] There exist $\alpha,\beta\in \mathbb R$ such that for every $x\in \HH$ and every $z=(z_1,z_2)\in\mathbb R^2$,
	$$
	F(z)+\alpha(\phi_1(x)-z_1)+\beta (\phi_2(x)-z_2)\ge0.
	$$
\end{itemize}
Moreover, 
\begin{enumerate}
    \item if $A_1$ is not bounded below and $A_2\in GL^+(\HH)$ then $\beta \ge0.$ Likewise, if $A_2$ is not bounded below and $A_1\in GL^+(\HH)$ then  $\alpha \ge0;$
    \item if either $\Theta=\left(\begin{array}{cc}
	0			 & 0  \\
	0		 &  \rho
	\end{array}\right)$ and $v_1>0$,
or $\Theta=\left(\begin{array}{cc}
	\rho			 & 0  \\
	0		 &  0
	\end{array}\right)$ and $v_2<0$ then $\alpha\ge 0$.
\end{enumerate}
 \end{lemma}

\medskip

In order to prove the above result  we need the following $\mathbb R^2$ version of Farkas' Theorem (see for example  \cite{polik2007survey}, \cite{de2003nonlinear}, \cite[section 6.10]{stoer2012convexity}):
	
	Let $F, h:\mathbb R^2\to\mathbb R$ be convex functions and suppose that  there exists $\bar{x}\in \mathbb{R}^2$ such that $h(\bar{x})\leq 0$. Then 
	$F(z)\ge 0$ for every $z\in\mathbb R^2$ such that $h(z)\le 0$ if and only if there exists $\lambda\ge 0$ such that $F(z)+\lambda h(z)\ge0$ for every  $z\in\mathbb R^2$.
\smallskip

\begin{proof}[of Lemma \ref{new S-lemma}]
	By Proposition \ref{Polyakdiminfinita}, $D:=\{(\phi_1(x),\phi_2(x))\,:\, x\in \HH\}$ is convex.
	Since $\Theta\ge0$, the set $\{z\in \mathbb R^2:\,F(z)<0\}$   is also convex.
    Moreover, by (i), $D\cap\{z\,:\,F(z)<0\}=\emptyset.$  Thus, we can separate these sets by a hyperplane in $\mathbb R^2$, i.e., there exist $\alpha,\beta,\gamma\in\mathbb R$ such that 
	\begin{align}
	\label{separacion convexos s-lemma1} z\in D &\Rightarrow \alpha z_1+\beta z_2+\gamma\ge 0   \qquad\textrm{ and},\\
	\nonumber F(z)<0 &\Rightarrow \alpha z_1+\beta z_2+\gamma< 0 .
	\end{align}
	Thus, $F(z)\ge 0 $ for every $z=(z_1,z_2)$ such that $ \alpha z_1+\beta z_2+\gamma\ge 0 $. 
	By the Farkas' Theorem, there exists $\lambda\ge0$ such that for every $z\in\mathbb R^2$, 
	$$
	F(z)-\lambda(\alpha z_1+\beta z_2+\gamma)\ge0.
	$$
	From this inequality and \eqref{separacion convexos s-lemma1} we conclude that,
	$$
	F(z)+\lambda\alpha(\phi_1(x)-z_1)+\lambda\beta(\phi_2(x)-z_2)=F(z)-\lambda(\alpha z_1+\beta z_2+\gamma)+\lambda(\alpha \phi_1(x)+\beta \phi_2(x)+\gamma)\ge 0,
	$$
	for every $z=(z_1,z_2)\in\mathbb R^2$ and every $x\in\HH$. The converse is straightforward.
	
	Moreover, 
		\begin{enumerate}
			\item Suppose that $A_1$ is not bounded below, $A_2\in GL^+(\HH)$ and $\beta<0$.
			 By \eqref{separacion convexos s-lemma1}, it holds that $z_2 \leq -\frac{\alpha}{\beta}z_1-\frac{\gamma}{\beta}$, for every $z\in D$.
			Then the set $D$ must be below  a line with finite slope. We will now prove that this is not possible.
			 Let $0<\varepsilon<\delta\frac{|\beta|}{|\alpha|}$, where $\delta>0$ is such that 
			 $\langle A_2x,x\rangle\ge\delta \|x\|^2$ for every $x\in\HH.$
			
			 Since $A_1$ is not bounded below,  given   $r>0$,  there exists $x\in \HH$ such that $\|x \|=r$ and $|\langle A_1x,x\rangle |<\epsilon r^2$. Then
			\begin{align*}
			    \delta r^2-\|a_2\|r-|b_2|\le\phi_2(x)&\le -\frac{\alpha}{\beta}\phi_1(x)-\frac{\gamma}{\beta}\\
			    &<\frac{|\alpha|}{|\beta|}(\varepsilon r^2+\|a_1\|r+b_1)+\frac{|\gamma|}{|\beta|}.
			\end{align*}
		Thus, for every $r$ we should have that 
		$$
		(\delta-\frac{|\alpha|}{|\beta|}\varepsilon) r^2-(\frac{|\alpha|}{|\beta|}\|a_1\|+\|a_2\|)r-|b_2|-\frac{|\alpha|}{|\beta|}b_1-\frac{|\gamma|}{|\beta|}<0.
		$$
				This is a contradiction because $\delta-\frac{|\alpha|}{|\beta|}\varepsilon>0$.
			
%
				\item Suppose now that  $\Theta=\left(\begin{array}{cc}
	0			 & 0  \\
	0		 &  \rho
    	\end{array}\right)$, $v_1>0$ and  $\alpha <0$.  By \eqref{separacion convexos s-lemma1}, it holds that $z_1 >-\frac{\beta}{\alpha}z_2 -\frac{\gamma}{\alpha}$, for every $z\in \mathbb R^2$ such that $ F(z)<0$.
			
			Since $\Theta=\left(\begin{array}{cc}
			0			 & 0  \\
			0		 &  \rho
			\end{array}\right)$ and $v_1>0$, then $F(z)=\rho z_2^2+v_1z_1+v_2z_2-t$. Therefore, $\{z:F(z)<0\}$ is the convex set determined by the parabola $z_1=-\frac{\rho}{v_1}z_2^2-\frac{v_2}{v_1}z_2+\frac{t}{v_1}$; so that it can not be on the right side of a straight line (for example, if $z_1<\min\{-\frac{\gamma}{\alpha},\frac{t}{v_1}\}$, then $(z_1,0)\in\{z : F(z)<0\}$ but does not satisfies \eqref{separacion convexos s-lemma1}).
		
			The other case follows similarly.
			\end{enumerate}	
			\qed
\end{proof}

\subsection{S-Procedure}
In \cite{polyak1998convexity}, Polyak gave several applications of his convexity theorem. Most of them can be extended to infinite dimensional spaces using our result. In this final subsection we briefly present as an example one of these extensions. Let $\HH$ be a real Hilbert space and $A_0, A_1, A_2 \in L(\HH).$ Given two quadratic forms
$$f_i(x)=\PI{A_ix}{x}, \  \ i=1,2  $$ 
in $\HH$ and $\alpha_1, \alpha_2 \in \mathbb{R}$; the problem is to characterize all $f_0(x)=\PI{A_0x}{x}, \alpha_0 \in \mathbb{R}$ such that
\begin{equation} \label{SP1}
f_0(x) \leq \alpha_0 \mbox{ for every } x \in \HH \mbox{ such that }  f_1(x) \leq \alpha_1, \ f_2(x) \leq \alpha_2.
\end{equation}

\begin{proposition} \label{teoSPro1} Let $\HH$ be a real Hilbert space, $3\le dim(\HH)\le \infty$. Suppose that there exist $\mu_1, \mu_2 \in \mathbb{R}, x^0 \in \HH$ such that
\begin{align}
\label{SP09}  & \mu_1A_1+\mu_2A_2 >0,\\
\label{SP10}  & f_1(x^0) < \alpha_1, \ f_2(x^0) < \alpha_2.
\end{align}
Then, \eqref{SP1} holds if and only if there exist $\tau_1\geq 0, \ \tau_2\geq 0$ such that
\begin{align} 
\label{SP11} & A_0\leq\tau_1A_1+\tau_2A_2,\\
\label{SP12} & \alpha_0\geq \tau_1\alpha_1+\tau_2 \alpha_2.
\end{align}
\end{proposition}
\begin{proof}
Consider
$$F:=\{f(x): x \in \HH\}, \ \ f(x):=(f_0(x),f_1(x),f_2(x)).$$
Then,  all the assumptions of Theorem \ref{convexidad Polyak1} hold; hence $F$ is convex. Then, the results follows using the same arguments as those found in the proof of \cite[Theorem 4.1]{polyak1998convexity}.
\end{proof}

Examples 4.1, 4.2 and 4.3 of \cite{polyak1998convexity} show that all the conditions of Theorem \ref{teoSPro1} are necessary.
 
\bigskip

A version of Proposition \ref{teoSPro1} where one of the inequalities $f_i(x) \leq \alpha_i$ is replaced by an equality can be proven with an extra condition. See also \cite[Proposition 4.1]{polyak1998convexity}.

\begin{proposition}
\label{propSPro2} Let $\HH$ be a real Hilbert space, $3\le dim(\HH)\le \infty$ and $\alpha_2 \not =0.$ Suppose that there exist $\mu_1, \mu_2 \in \mathbb{R}$ satisfying \eqref{SP09}, $x^0 \in \HH$ such that
\begin{align}
& f_1(x^0) < \alpha_1, \ f_2(x^0) = \alpha_2.
\end{align}
Then, 
$$f_0(x) \leq \alpha_0 \mbox{ for every } x \in \HH \mbox{ such that }  f_1(x) \leq \alpha_1, \ f_2(x)= \alpha_2,$$
if and only if there exists $\tau_1\geq 0$ such that \eqref{SP11} and \eqref{SP12} hold.
\end{proposition} 


\section{Conclusions}
An important result due to Polyak \cite{polyak1998convexity} states that the joint image of two non-homogeneous quadratic forms defined on $\mathbb R^n$ is a convex closed set of $\mathbb R^2$.
This class of result has many applications, for instance to S-lemma type results.

In this article we extend the convexity part of Polyak's 
result to an arbitrary infinite dimensional real Hilbert space $\HH$, see Theorem \ref{convexidad Polyak1} and Proposition  \ref{Polyakdiminfinita}.

We present examples involving diagonal operators showing that the closedness part of Polyak's theorem does not hold on infinite dimensional spaces for quadratic forms determined by (compact or invertible) positive definite operators.
Moreover, we show that if $A_1$ is a compact operator on $\HH$ with 0 in its numerical range and $A_2$ is a positive and invertible definite operator, then the joint image of two non necessarily homogeneous quadratic forms determined by $A_1$ and $A_2$ is closed, see Proposition \ref{propclosednohomogeneo}.

For further research, it would be interesting to find necessary and sufficient conditions that allow to prove the closedness part of Polyak's theorem in the infinite dimensional setting.

\section*{Acknowledgments}
Maximiliano Contino was supported by CONICET PIP 0168. Guillermina Fongi was supported by PICT 2017 0883. Santiago Muro was supported by ANPCyT-PICT 2018-04250 and CONICET-PIP 11220130100329CO.

\goodbreak

\end{document}